\documentclass{amsart}

\usepackage[pdfusetitle, colorlinks=true]{hyperref}
\usepackage{color}

\usepackage{amsmath, amsthm, amssymb, amsfonts, mathtools}
\usepackage{graphicx,epsfig,color}
\usepackage{scrpage2}
\usepackage{exscale}

\theoremstyle{plain}

\newcommand{\VK}{k}

\newtheorem{theorem}{Theorem}[section]

\newtheorem{lemma}[theorem]{Lemma}

\newtheorem{corollary}[theorem]{Corollary}

\theoremstyle{definition}

\newtheorem{definition}[theorem]{Definition}

\theoremstyle{remark}
\newtheorem{remark}[theorem]{Remark}

\newtheorem{example}[theorem]{Example}

\newtheorem*{jjj}{Remark}

\DeclareSymbolFont{AMSb}{U}{msb}{m}{n}
\DeclareMathSymbol{\N}{\mathalpha}{AMSb}{"4E}
\DeclareMathSymbol{\R}{\mathalpha}{AMSb}{"52}
\DeclareMathSymbol{\Z}{\mathalpha}{AMSb}{"5A}
\DeclareMathSymbol{\D}{\mathalpha}{AMSb}{"44}
\DeclareMathSymbol{\s}{\mathalpha}{AMSb}{"53}

\newcommand{\Q}{\ensuremath{\mathbb{Q}}}

\DeclareMathOperator{\vol}{vol}

\DeclareMathOperator{\supp}{supp}

\DeclareMathOperator{\m}{m}

\DeclareMathOperator{\diam}{diam}

\newcommand{\T}{\mathcal{T}}
\DeclareMathOperator{\sign}{sign}

\title{The Heintze-Karcher inequality for metric measure spaces}\thanks{\textit{2010 Mathematics Subject classification}. Primary 53C21 30L99, Keywords: curvature-dimension condition, mean curvature, optimal transport, comparison geometry}
\author{Christian Ketterer}
\address{University of Toronto}
\thanks{
The author is funded by the Deutsche Forschungsgemeinschaft (DFG, German Research Foundation) -- Projektnummer 396662902.
}
\email{ckettere@math.toronto.edu.}
\begin{document}
\maketitle
\begin{abstract} In this note we prove  the Heintze-Karcher inequality  in the context of essentially non-branching metric measure spaces satisfying a lower Ricci curvature bound in the sense of Lott-Sturm-Villani. The proof is based on the needle decomposition technique for metric measure spaces introduced by Cavalletti-Mondino. Moreover, in the  class of $RCD$ spaces with positive curvature the equality case is characterized.
\end{abstract}
\tableofcontents
\section{Introduction}
The Heintze-Karcher theorem is a classical volume comparison result in Riemannian geometry \cite{heintzekarcher} (see also \cite{maeda}). It states that the one sided tubular neighborhood of a two sided $C^2$ hypersurface $S$ in an $n$-dimensional Riemannian manifold $M$ is bounded by a surface integral over $S$ involving  the mean curvature, a lower bound for the Ricci curvature and  an upper bound of the dimension $n$.
The original proof is based on Jacobi field computations and similar estimates were obtained in \cite{peralesheka} applying refined Laplace comparison estimates for manifolds with boundary.
When $M$ is equipped with a smooth measure $\Phi \m$, $\Phi\in C^{\infty}(M)$, a generalisation was proven by Bayle in \cite{baylethesis} (see also \cite{morganheka}) where Ricci  curvature is replaced by the Bakry-Emery $N$-Ricci curvature, the mean curvature with generalized mean curvature and the volume of $S$ with the weighted volume. 
The Heintze-Karcher estimate found numerous applications in Riemannian geometry (e.g. \cite{milman, peralesheka, mnwillmore}).

In this note we prove  Heintze and Karcher's theorem in the context of essentially non-branching metric measure spaces with finite measure satisfying a lower Ricci curvature bound in the sense of Lott-Sturm-Villani \cite{stugeo2, lottvillani}. 
More precisely, we consider  an essentially nonbranching $CD(K,N)$ space $(X,d,\m)$ for $K\in \R$ and $N\in [1,\infty)$ with finite measure $\m$
and a generalized hypersurface $S$ that is the boundary of a Borel subset $\Omega\subset X$ and satisfies $\m(S)=0$.
For this setup one can introduce a notion of mean curvature for $S$ using the $1D$-localisation technique for $1$-Lipschitz functions established by Cavalletti-Mondino \cite{cavmon, cav-mon-lapl-18} (see also previous work by Klartag, Cafarelli, Feldman and McCann \cite{klartag, cfm}).

Let us describe our approach. A precise construction is given in Section \ref{sec:mean}. Associated to $S=\partial \Omega$ we consider the signed distance function $d_S$ that is $1$-Lipschitz for $CD(K,N)$ spaces. Then, the localisation technique provides a measurable decomposition of the space into geodesic segments $\gamma_{\alpha}:(a_\alpha,b_\alpha)\rightarrow X$, $\alpha\in Q$, and a disintegration $\m= \int_Q\m_{\alpha} \mathfrak q(d\alpha)$ of the measure $\m$ with a quotient space $(Q,\mathfrak q)$.  The measure $\m_\alpha$ is supported on $\mbox{Im}(\gamma_{\alpha})$ and has a semiconcave density $h_{\alpha}$ w.r.t. $\mathcal H^1$ for $\mathfrak q$-a.e. $\alpha\in Q$. 
One can define the outer mean curvature in any point $p\in S$ satisfying $\gamma_\alpha(r)=p$ for some $\alpha\in Q$ and $r\in (a_\alpha,b_\alpha)$ such that $h_\alpha$ exists via $$\frac{d^+}{dr}\log h_{\alpha}(r)=\lim_{h\downarrow 0}h^{-1}\left(\log h_\alpha(r+h)-\log h_\alpha(r)\right)=:H^+(p).$$ 
Similar one defines the inner mean curvature in such a point as $-\frac{d^-}{dr}\log h_{\alpha}(r)=H^-(p)$. Then, the mean curvature in $p$ is defined as
\begin{align*}
H(p)=\max\left\{ H^+(p),-H^-(p)\right\}.
\end{align*}
This notion of generalized mean curvature will  be sufficient to prove the Heintze-Karcher estimate.
In smooth context,  $H^+=-H^-$ and $H$ will coincide with (minus) the classical mean curvature (Remark \ref{rem:classicalH}). Hence, our sign convention will be that the outer mean curvature of the boundary of a convex body is positive. The decomposition also allows to define a surface measure $\m_S$ that is supported on points $p\in S$ such that $\gamma_\alpha(r)=p$ for some $\gamma_\alpha$ and $r\in (a_\alpha,b_\alpha)$ such that $h_\alpha$ exists (Definition \ref{def:surfacemeasure}). Again in smooth context this will coincide with the classical notion (Remark \ref{rem:smooth}).

The main theorem of this note is the following.
\begin{theorem}\label{main}
Let $(X,d,\m)$ be an essentially non-branching metric measure space with $\m(X)<\infty$ satisfying the condition $CD(K,N)$ for $K\in \R$ and $N\in (1,\infty)$. Let $\Omega \subset X$ be a Borel subset such that $\m(\partial \Omega)=0$ and  $\partial \Omega =:S$ has finite outer curvature (see Definition \ref{def:meancurvature}).
Then
\begin{align}\label{ineq:hk0}
\m(S^+_t)=\m(B_t(\Omega)\backslash \Omega)\leq  \int\int_0^t J_{H^+(p),K,N}(r) dr d\m_S(p) \ \forall t\in (0, D].
\end{align}
where $B_t(\Omega)=\{x\in X: \exists p\in \Omega: d(x,p)<t\}$ and $D:=\diam_X$. 
$J_{H,K,N}$ is the Jacobian function (Definition \ref{jacobian}). 
\smallskip

If $S$ has finite curvature (Definition \ref{def:meancurvature}), then
\begin{align}\label{ineq:hk}
\m(X)\leq  \int\int_{[-D,D]} J_{H(p),K,N}(r) dr d\m_S(p).
\end{align}  
\end{theorem}
%
%
\begin{remark}
The regularity assumption "finite outer curvature" in the sense of Definition \ref{def:meancurvature} is necessary even in smooth context for the validitiy of the statement above as surfaces with corners show.

\end{remark}
\begin{jjj}
To keep the presentation short 
we only consider $CD(K,N)$ spaces with $N>1$.
\end{jjj}

Theorem \ref{main} is a  generalisation of the Heintze-Karcher theorem and  specializes to the classical statement in smooth context.
The class of essentially nonbranching $CD$ spaces includes for instance finite dimensional $RCD$ spaces, weighted Finsler manifolds with lower bounds for their $N$-Ricci tensor and finite dimensional Alexandrov spaces.

We can assume an upper bound $H_0\in \R $ for the mean curvature and obtain the following Corollary.

\begin{corollary}\label{corollary} Let $S$ be as in Theorem \ref{main} with finite curvature. If  $H_S\leq H_0$, it follows
\begin{align*}
\m(X)\leq \m_S(S) \int_{[-D,D]} J_{H_0,K,N}(r)dr.
\end{align*}
If $H_S\leq 0$ and $K\geq 0$ then 
$
\m(X)\leq \diam_X \m_S(S).
$
\end{corollary}
Let 
$\pi_\kappa$ be the diameter of a simply connected space form $\mathbb S_k^2$ of constant curvature $\kappa$, i.e.
$\pi_\kappa=
 \infty \ \textrm{ if } \kappa\le 0$ or $\pi_\kappa
=\frac{\pi}{\sqrt \kappa}\   \textrm{ if } \kappa> 0
$.
In the case of $K>0$ the generalized Heintze-Karcher estimate also takes the following form.
\begin{corollary}\label{corollary1}
Let $(X,d,\m)$ be a metric measure space and let $S$ be as in the Theorem \ref{main} with finite  curvature. Assume $K>0$.
Then
\begin{align*}
\m(X)\leq \int_0^{\pi_{K/(N-1)}}\sin^{N-1}\left(\scriptstyle{\sqrt{\frac{K}{N-1}}} r\right) dr \int \left(\textstyle{\frac{K}{N-1}+\left(\frac{H(p)}{N-1}\right)^2}\right)^{\frac{N-1}{2}} d\m_S(p).
\end{align*}
If $N\in \N$, we obtain
\begin{align*}
\m(X)\leq \frac{\vol(\mathbb S^{N}_{K/(N-1)})}{\vol(\mathbb{S}^{N-1}_1)} \int \left(\textstyle{\frac{K}{N-1}+\left(\frac{H(p)}{N-1}\right)^2}\right)^{\frac{N-1}{2}} d\m_S(p).
\end{align*}
\end{corollary}
\begin{remark}
The second estimate in the previous corollary also appears in the work of Heintze-Karcher \cite[2.2 Theorem]{heintzekarcher}.
\end{remark}

Recall that the class of $CD$ spaces can be enforced naturally to the class of $RCD$ spaces by requiring that the space of Sobolev functions is a Hilbert space.
For positive $K$ and in the context of $RCD(K,N)$ spaces with $N\in [1,\infty)$ the following theorem characterizes the equality case in \eqref{ineq:hk} and Corollary \ref{corollary1}.
\begin{theorem}\label{main2}
Let $(X,d,\m)$ be a metric measure space that satisfies the condition $RCD(K,N)$ for $K>0$ and $N\in (1,\infty)$ and let $S$ be as in Theorem \ref{main} with finite  curvature.

Equality in \eqref{ineq:hk} of Theorem \ref{main} or in Corollary \ref{corollary1} holds if and only if there exists an $RCD(N-2,N-1)$ space $(Y,d_Y,\m_Y)$ such that $X$ is a spherical suspension over $Y$:
\begin{align*}
X=I_{K,N}\times_{\sin\left(\scriptscriptstyle{\sqrt{\frac{K}{N-1}}}\cdot\right)}^{N-1} Y
\end{align*}
where $I_{K,N}=\left(\left[0,{\pi_{K/(N-1)}}\right], 1_{\left[0,\scriptstyle{\pi_{K/(N-1)}}\right]}\sin_{K/(N-1)}^{N-1}\mathcal L^1\right)$
and $S$ is a constant mean curvature surface in $X$. More precisely, $S$ is a sphere centered at one of the poles of $X$.
Here, we use the warped product notation for spherical suspensions (compare with the exposition in Subsection \ref{subsec:warped})
\end{theorem}
%
The rest of this note is organized as follows.  
In section 2 we briefly recall some facts about optimal transport and the Wasserstein space of a metric measure space, the curvature-dimension condition for essentially non-branching metric measure spaces, warped products, the Cavalletti-Mondino isoperimetric comparison and a general disintegration theorem for measure spaces.

In section 3 we explain the $1D$-localisation technique by Cavalletti-Mondino and how it applies in the context of essentially non-branching metric measure spaces satisfying $CD(K,N)$.

In section 4 we prove a simple comparison result in $1D$ that follows from Sturm's comparison theorem. 

In section 5 we introduce the signed distance function for a set $S$ that arises as boundary of a Borel set in a metric measure space. We describe briefly how the localisation technique applies for this functions. This structure allows us to define the mean curvature of $S$ and the generalized surface measure $\m_S$. We also show that these notions coincide with the classical ones in the context of weighted Riemannian manifolds.

In section 6 we prove the main theorems of this note.

 \subsection{Acknowledgments} 
The author want to thank Robert McCann, Vitali Kapovitch and Robert Haslhofer for valuable discussions about topics related to this work and Samuel Borza for pointing out the reference \cite{maeda}. Moreover, the author is grateful to Robert McCann for reading carefully an early version of this article. The author is also very grateful to the anonymous referee for giving useful comments and remarks (especially Remark \ref{rem:ballcondition}) that helped to improve the final version of this note.

\section{Preliminaries}

\subsection{Curvature-dimension condition}
In this subsection we recall some facts about optimal transport, the geometry of the Wasserstein space and synthetic Ricci curvature bounds. For more details we refer to \cite{viltot}. We also assume familarity with calculus on length metric spaces. For details we refer to \cite{bbi}.

Let $(X,d)$ be a metric space.
A rectifiable constant  speed curve $\gamma:[a,b]\rightarrow X$ is a \textit{geodesic} if $\mbox{L}(\gamma)=d(\gamma(a),\gamma(b))$ where $\mbox{L}$ is the induced length functional.
We say $(X,d)$ is a \textit{geodesic metric space} if for any pair $x,y\in X$ there exists a geodesic between $x$ and $y$.
The set of all constant speed geodesics $\gamma:[0,1]\rightarrow X$ is denoted with $\mathcal G^{[0,1]}(X)=:\mathcal{G}(X)$ and equipped with the topology of uniform convergence. For $t\in [0,1]$ we write $e_t:\gamma\in \mathcal G(X)\mapsto \gamma(t)$ for the evaluation map.

A set $F\subset \mathcal{G}(X)$ is said to be non-branching if and only if for any two geodesics $\gamma^1,\gamma^2\in \mathcal G(X)$ the following holds. 
$
\mbox{If }\gamma^1|_{[0,\epsilon)}\equiv \gamma^2|_{[0,\epsilon)} \mbox{ for some }\epsilon >0\ \mbox{ then }\ \gamma^1\equiv \gamma^2.
$

The set of Borel probability measures $\mu$ on $(X,d)$ such that $\int_Xd(x_0,x)^2d\mu(x)<\infty$ for some $x_0\in X$ is denoted $\mathcal P^2(X)$. For any pair $\mu_0,\mu_1\in \mathcal{P}^2(X)$ 
we denote with $W_2(\mu_0,\mu_1)$ the \textit{$L^2$-Wasserstein distance}.
%
%
We call the metric space $(\mathcal{P}^2(X),W_2)$ the \textit{$L^2$-Wasserstein space} of $(X,d)$.
The subspace of probability measures with bounded support is denoted with
$\mathcal{P}^2_b(X)$. 

\begin{definition}
A \textit{metric measure space} is a triple $(X,d,\m)=:X$ where $(X,d)$ is a complete and separable metric space and $\m$ is a locally finite Borel measure.
\end{definition}
The space of $\m$-absolutely continuous probability measures in $\mathcal{P}^2(X)$ is denoted by $\mathcal{P}^2(X,\m)$. Similar we define $\mathcal P_b^2(X,\m)$.

Any geodesic $(\mu_t)_{t\in [0,1]}$ in $(\mathcal P^2(X,\m),W_2)$ can be lifted to a measure $\Pi\in \mathcal P(\mathcal G(X))$ such that $(e_t)_{\#}\Pi=\mu_t$. We call such a measure $\Pi$ a {\it dynamical optimal plan}.

A metric measure space $(X,d,\m)$ is said to be essentially non-branching if for any two measures $\mu_0,\mu_1\in \mathcal P^2(X,\m)$ any dynamical optimal plan $\Pi$ is concentrated on a set of non-branching geodesics.

\begin{example}
Given a Riemannian manifold $(M,g)$  and a measure $\m=\Psi\vol_g$ for $\Psi\in C^{\infty}(M)$ we call the triple $(M,g,\m)$ a weighted Riemannian manifold.
\end{example}
\begin{definition}
For $\kappa\in \mathbb{R}$ we define $\cos_{\kappa}:[0,\infty)\rightarrow \mathbb{R}$ as the solution of 
\begin{align*}
v''+\kappa v=0 \ \ \ v(0)=1 \ \ \& \ \ v'(0)=0.
\end{align*}
$\sin_{\kappa}$ is defined as solution of the same ODE with initial value $v(0)=0 \ \&\ v'(0)=1$. 
%
For $K\in \mathbb{R}$, $N\in (0,\infty)$ and $\theta\geq 0$ we define the \textit{distortion coefficient} as
\begin{align*}
t\in [0,1]\mapsto \sigma_{K,N}^{(t)}(\theta)=\begin{cases}
                                             \frac{\sin_{K/N}(t\theta)}{\sin_{K/N}(\theta)}\ &\mbox{ if } \theta\in [0,\pi_{K/N}),\\
                                             \infty\ & \ \mbox{otherwise}.
                                             \end{cases}
\end{align*}
Note that $\sigma_{K,N}^{(t)}(0)=t$.
Moreover, using the convention $0\cdot \infty =0$ for $K\in \mathbb{R}$, $N\in [1,\infty)$ and $\theta\geq 0$ the \textit{modified distortion coefficient} is defined as
\begin{align*}
t\in [0,1]\mapsto \tau_{K,N}^{(t)}(\theta)=\begin{cases}
                                            \theta\cdot\infty \ & \mbox{ if }K>0\mbox{ and }N=1,\\
                                            t^{\frac{1}{N}}\left[\sigma_{K,N-1}^{(t)}(\theta)\right]^{1-\frac{1}{N}}\ & \mbox{ otherwise}.
                                           \end{cases}\end{align*}
\end{definition}

\begin{definition}[\cite{stugeo2,lottvillani}] An essentially non-branching metric measure space $(X,d,\m)$ satisfies the \textit{curvature-dimension condition} $CD(K,N)$ for $K\in \mathbb{R}$ and $N\in [1,\infty)$ if for every $\mu_0,\mu_1\in \mathcal{P}_b^2(X,\m)$ 
there exists a dynamical optimal coupling $\Pi$ between $\mu_0$ and $\mu_1$ such that for all $t\in (0,1)$
$$
\rho_t(\gamma_t)^{-\frac{1}{N}}\geq \tau_{K,N}^{(1-t)}(d(\gamma_0,\gamma_1))\rho_0(\gamma_0)^{-\frac{1}{N}}+\tau_{K,N}^{(t)}(d(\gamma_0,\gamma_1))\rho_1(\gamma_1)^{-\frac{1}{N}}
$$
$\mbox{for }\Pi\mbox{-a.e.}\ \gamma\in \mathcal{G}(X)$ and 
for all $t\in [0,1]$ where $(e_t)_{\#}\pi=\rho_t\m$.
%
%
\end{definition}
\begin{example}
The metric measure space $(M,d_g,\Psi \vol_g)$ associated to a weighted Riemannian manifold $(M,g,\Psi\vol_g)$ for $\Psi\in C^{\infty}(M)$ satisfies the condition $CD(K,N)$, $K\in \R$, $N\in [1,\infty)$, if and only if $M\backslash \partial M$ is geodesically convex and the Bakry-Emery $N$-Ricci tensor is bounded from below by $K$ on $M\backslash \partial M$.
\end{example}
\begin{definition}[\cite{agsriemannian, giglistructure, erbarkuwadasturm, cavmil, amsnonlinear}]\label{def:rcd}
The Riemannian curvature-dimension condition $RCD(K,N)$ for $K\in \R$ and $N\in [1,\infty)$ is defined as the condition $CD(K,N)$ together with the property  that the associated Sobolev space $W^{1,2}(X)$ is a Hilbert space.
\end{definition}
For a brief overview of the historical development of the previous definition we also refer to the preliminaries of \cite{Kap-Ket-18}.
\subsection{Warped products}\label{subsec:warped}
For $K>0$ and $N> 1$ the $1$-dimensional model space is $$I_{K,N}=\left(\left[0,{\pi_{K/(N-1)}}\right], 1_{\left[0,\scriptstyle{\pi_{K/(N-1)}}\right]}\sin_{K/(N-1)}^{N-1}\mathcal L^1\right)$$
where $[0,\scriptstyle{\pi_{K/(N-1)}}]$ is equipped with the restriction of the standard metric $|\cdot|$ on $\R$. The metric measure space
$I_{K,N}$ satisfies $CD(K,N)$ \cite[Example 1.8]{stugeo2}. 

Let $(M,g,\m)=M$ be a weighted Riemannian manifold with $\m={\Phi}\vol_g$ and $\Phi\in C^{\infty}(M\backslash \partial M)$. The warped product
$I_{K,N}\times_f^{N-1}M$ between $I_{K,N}$ and $M$ w.r.t. $f:I_{K,N}\rightarrow [0,\infty)$ is defined as the metric completion of the weighted Riemannian manifold
$\left(I_{K,N}\times M, h, \m_C\right)$
where $h=\langle\cdot,\cdot\rangle^2+ f^2 g$ and $\m_C=f^{N-1} \mathcal L^1|_{I_{K,N}} \otimes \m$. In \cite{ketterer1} it was proved that if the warping function $f$ satisfies 
\begin{align*}
f''+ \frac{K}{N-1}f\leq 0\ \mbox{ and }\  (f')^2+\frac{K}{N-1}f^2\leq L \mbox{ on } I_{K,N}
\end{align*}
and $(M,d_g,\m)$ satisfies $CD(L(N-2),N-1)$ then $I_{K,N}\times^{N-1}_{f} M$ satisfies $CD(K,N)$. This applies in particular when $f=\sin_{K/(N-1)}$ and $L=1$. Then the corresponding warped product is a spherical suspension. For instance, we can choose $M=I_{N-2,N-1}$. If $n\in \N$ we can choose 
$M=\mathbb{S}^{n-1}_{1}$ and we get that
$
I_{K,n}\times^{n-1}_{\sin_{K/(n-1)}}\mathbb S^{n-1}_1=\mathbb{S}_{K/(n-1)}^n.
$

More generally, one can define warped products in the context of metric measure spaces. In \cite{ketterer2} it was proved that 
$
I_{K,N-1}\times_{\sin_{{K}/{N-1}}}^{N-1} Y
$ 
satisfies the condition $RCD(K,N)$ if and only if $Y=(Y,d_Y,\m_Y)$ satisfies the condition $RCD(N-2,N-1)$.
\subsection{Isoperimetric profile}
Let $(X,d,\m)$ be a metric measure space such that $\m$ is finite, and let $A\subset X$. Denote $A_{\epsilon}=B_{\epsilon}(A)$ the $\epsilon$-tubular neigborhood of $A$. We also set $A^+_{\epsilon}=A_{\epsilon}\backslash A$. The (outer) Minkowski  content $\m^+(A)$ of $A$ is defined by
\begin{align*}\m^+(A)=
\limsup_{\epsilon\rightarrow 0}\frac{\m(A^+_{\epsilon})}{\epsilon}.
\end{align*}
The {\it isoperimetric profile function} $\mathcal{I}_{(X,d,\m)}: [0,1]\rightarrow [0,\infty)$ of $(X,d,\m)$ is defined as follows. Let $\bar \m=\m(X)^{-1}\m$ be the probability measure proportional to $\m$. Then, we set
\begin{align*}
\mathcal I_{(X,d,\m)}(v):= \inf\left\{\bar \m^+(A): A\subset X\mbox{ Borel }, \bar \m(A)=v\right\}.
\end{align*}
Let $K>0$. The model isoperimetric profile for spaces with Ricci curvature bigger or equal than $K$ and dimension bounded above by $N>1$  is given by 
$
\mathcal I_{K,N,\infty}(v):= \mathcal I_{I_{K,N-1}}(v), \ v\in [0,1],
$
where $I_{K,N}$ is again the $1$-dimensional model space that was introduced in the previous section.

The following theorem is one of the main results in \cite{cavmon} and generalizes the Levy-Gromov isoperimetric inequality for Riemannina manifolds.
\begin{theorem}[Cavalletti-Mondino]\label{th:cavmon}
Let $(X,d,\m)$ be an essentially non-branching $CD(K,N)$ space for $K>0$ and $N> 1$. 
Then for every Borel set $E\subset X$ it holds that 
\begin{align}\label{ineq:levigromov}
{\bar \m^+(E)}\geq \mathcal I_{K,N,\infty}\left({\bar \m(E)}\right).
\end{align}
If $(X,d,\m)$ satisfies the condition $RCD(K,N)$ and there exists $A\subset X$ with equality in \eqref{ineq:levigromov} then
\begin{align*}
X=I_{K,N}\times_{\sin_{K/(N-1)}}^{N-1} Y
\end{align*}
for some $RCD(N-2,N-1)$ space $(Y,d_{Y},\m_{Y})$.
Moreover, $\bar A$ is a closed geodesic ball of volume $\m(A)$ centered at one of the origins of $I_{K,N}\times^{N-1}_{\sin_{K/(N-1)}}Y$.
%
\end{theorem}
\subsection{Disintegration of measures} For further details about the content of this section we refer to \cite[Section 452]{fremlin}.

Let $(R,\mathcal R)$ be a measurable space, and let $\mathfrak Q:R\rightarrow Q$ be a map for a set $Q$. One can equip $Q$ with the  $\sigma$-algebra $\mathcal Q$ that is induced by $\mathfrak Q$ where  $B\in \mathcal Q$  if $\mathfrak Q^{-1}(B)\in \mathcal{R}$. Given a probability measure $\m$ on $(R,\mathcal R)$, one can define a probability measure $\mathfrak q$ on $Q$ via the pushforward $\mathfrak Q_{\#} \m=: \mathfrak q$.
\begin{definition}
A disintegration of $\m$ that is consistent with $\mathfrak Q$ is a map $(B,q)\in \mathcal R \times Q\mapsto \m_{\alpha}(B)\in [0,1]$ such that the following holds
\begin{itemize}
\item $\m_{\alpha}$ is a probability measure on $(R,\mathcal R)$ for every $\alpha\in Q$,
\item $\alpha\mapsto \m_{\alpha}(B)$ is $\mathfrak q$-measurable for every $B\in \mathcal R$,
\end{itemize}
and for all $B\in \mathcal R$ and $C\in \mathcal Q$ the {\it consistency condition} 
$$
\m(B\cap \mathfrak Q^{-1}(C))=\int_C \m_{\alpha}(B) \mathfrak q(d\alpha)
$$
holds. We use the notation $\{\m_{\alpha}\}_{\alpha\in Q}$ for such a disintegration. We call the measures $\m_{\alpha}$ {\it conditional probability measures}.

A disintegration $\{\m_{\alpha}\}_{\alpha\in Q}$ is called strongly consistent with respect $\{\mathfrak Q^{-1}(\alpha)\}_{\alpha\in Q}$ if for $\mathfrak q$-a.e. $\alpha$ we have $\m_{\alpha}(\mathfrak Q^{-1}(\alpha))=1$.  
\end{definition}
\begin{theorem}
Assume that $(R,\mathcal R, \m)$ is a countably generated probabilty space and $R=\bigcup_{\alpha\in Q}R_{\alpha}$ is a partition of $R$. Let $\mathfrak Q: R\rightarrow Q$ be the quotient map associated to this partition, that is $\alpha =\mathfrak Q(x)$ if and only if $x\in R_{\alpha}$ and assume the corresponding quotient space $(Q,\mathcal Q)$ is a Polish space.

Then, there exists a strongly consistent disintegration $\{\m_{\alpha}\}_{\alpha\in Q}$ of $\m$ w.r.t. $\mathfrak Q:R\rightarrow Q$ that is unique in the following sense: if $ \{\m'_{\alpha}\}_{\alpha\in Q}$ is another consistent disintegration of $\m$ w.r.t. $\mathfrak Q$ then $\m_{\alpha}=\m'_{\alpha}$ for $\mathfrak q$-a.e. $\alpha \in Q$.
\end{theorem}
\section{$1$-localisation of generalized Ricci curvature bounds.}
\noindent
In this section we will briefly recall the localisation technique introduced by Cavalletti and Mondino. The presentation follows Section 3 and 4 in \cite{cavmon}. We assume familarity with basic concepts in optimal transport (for instance \cite{viltot}).

Let $(X,d,\m)$ be a locally compact metric measure space that is essentially nonbranching. We  assume that $\supp\m =X$ and $\m(X)<\infty$.

Let $u:X\rightarrow \mathbb{R}$ be a $1$-Lipschitz function. Then 
\begin{align*}
\Gamma_u:=\{(x,y)\in X\times X : u(x)-u(y)=d(x,y)\}
\end{align*}
is a  $d$-cyclically monotone set, and one defines $\Gamma_u^{-1}=\{(x,y)\in X\times X: (y,x)\in \Gamma_u\}$. The union $\Gamma\cup \Gamma^{-1}$ defines a relation $R_u$ on $X\times X$, and $R_u$ induces the {\it transport set with endpoints} 
$$\mathcal T_{u,e}:= P_1(R_u\backslash \{(x,y):x=y\in X\})\subset X$$ 
where $P_1(x,y)=x$. For $x\in \T_{u,e}$ one defines $\Gamma_u(x):=\{y\in X:(x,y)\in \Gamma_u\}, $
and similar $\Gamma_u^{-1}(x)$ and $R_u(x)$. Since $u$ is $1$-Lipschitz, 
 $\Gamma_u, \Gamma_u^{-1}$ and $R_u$ are closed as well as $\Gamma_u(x), \Gamma_u^{-1}(x)$ and $R_u(x)$.

The {\it forward} and {\it backward branching points} are defined respectively as
\begin{align*}
A_{+}\!:=\!\{x\in \mathcal T_{u,e}: \exists z,w\in \Gamma_u(x) \mbox{ \& } (z,w)\notin R_u\}, \ 
A_{-}\!:=\!\{x\in \mathcal T_{u,e}: \exists z,w\in  \Gamma_u(x)^{-1} \mbox{ \& } (z,w)\notin R_u\}.
\end{align*}
Then one considers the {\it (nonbranched) transport set} $\mathcal T_u:=\mathcal T_{u,e}\backslash (A_+ \cup A_-)$ and the {\it (nonbrached) transport relation} that is the restriction of $R_u$ to $\mathcal T_u\times \mathcal T_u$. 

As showed in \cite{cavmon}
$\T_{u,e}$, $A_{+}$ and $A_-$ are $\sigma$-compact, and $\T_u$ is a  Borel set.
In \cite{cavom} Cavalletti shows that the restriction of $R_u$ to $\mathcal T_u\times \mathcal T_u$ is an equivalence relation. 
Hence, from $R_u$ one obtains a partition of $\mathcal T_u$ into a disjoint family of equivalence classes $\{X_{\alpha}\}_{\alpha\in Q}$. There exists a measurable section $s:\T_u\rightarrow \T_u$, that is $s(x) \in R_u(x)$, and  $Q$ can be identified with the image of $\T_u$ under $s$. Every $X_{\alpha}$ is isometric to an interval $I_\alpha\subset\mathbb{R}$ via an isometry $\gamma_{\alpha}:I_\alpha \rightarrow X_{\alpha}$ where $\gamma_\alpha$ is parametrized such that $d(\gamma_\alpha(t), s(\gamma_\alpha(t)))=t$, $t\in I_\alpha$, for the section $s$ before. The map $\gamma_{\alpha}:I_\alpha\rightarrow X$ extends to a geodesic also  denoted $\gamma_{\alpha}$ and defined on the closure $\overline{I}_{\alpha}$ of $I_{\alpha}$. We set $\overline{I}_{\alpha}=[a(X_{\alpha}),b(X_{\alpha})]$.

The index set $Q$ can also be written as
\begin{align*}
Q=\bigcup_{n\in \N}Q_n \mbox{ where } Q_n=u^{-1}(l_n)\mbox{ and }l_n\in \Q 
\end{align*}
and $Q_i\cap Q_j$ for $i\neq j$, and $Q$ is equipped with the induced measurable structure \cite[Lemma 3.9]{cavmon}.
Then, the quotient map $\mathfrak Q:\T_u\rightarrow Q$ is measurable, and we set $\mathfrak q:= \mathfrak Q_{\#}\m$.

\begin{theorem}
Let $(X,d,\m)$ be a compact geodesic metric measure space with $\supp\m =X$ and $\m$ finite. Let $u:X\rightarrow \mathbb{R}$ be a $1$-Lipschitz function,  let $(X_{\alpha})_{\alpha\in Q}$ be the induced partition of $\mathcal T_u$ via $R_u$, and let $\mathfrak Q: \T_u\rightarrow Q$ be the induced quotient map as above.
Then, there exists a unique strongly consistent disintegration $\{\m_{\alpha}\}_{\alpha\in Q}$ of $\m|_{\T_u}$ w.r.t. $\mathfrak Q$. 
\end{theorem}
Now, we assume that $(X,d,\m)$ is an essentially non-branching $CD(K,N)$ space for $K\in \R$ and $N> 1$. 
The following lemma is Theorem 3.4 in \cite{cavmon}.
\begin{lemma}\label{somelemma}
Let $(X,d,\m)$ be an essentially non-branching $CD(K,N)$ space for $K\in \R$ and $N\in (1,\infty)$ with $\supp \m=X$ and $\m(X)<\infty$.
Then, for any $1$-Lipschitz function $u:X\rightarrow \R$, it holds $\m(\T_{u,e}\backslash \T_u)=0$.
\end{lemma}
The initial and final points are defined as follows
\begin{align*}
\mathfrak a \!:=\! \left\{ x\in \T_{u,e}: \Gamma^{-1}_u(x)=\{x\}\right\}, \ 
\mathfrak b \!:=\! \left\{ x\in \T_{u,e}: \Gamma_u(x)=\{x\}\right\}.
\end{align*}
In \cite[Theorem 7.10]{cavmil} it was proved that under the assumption of the previous lemma there exists $\hat Q\subset Q$ with $\mathfrak q(Q\backslash \hat Q)=0$ such that for $\alpha\in \hat Q$ one has $\overline{X_\alpha}\backslash \T_u\subset \mathfrak a\cup \mathfrak b$. In particular, for $\alpha\in \hat Q$ we have
\begin{align}\label{somehow}
R_u(x)=\overline{X_\alpha}\supset X_\alpha \supset (R_u(x))^{\circ} \ \ \forall x\in \mathfrak Q^{-1}(\alpha)\subset \T_u.
\end{align}
where $(R_u(x))^\circ$ denotes the relative interior of the closed set $R_u(x)$. 
\begin{theorem}\label{th:1dlocalisation}
Let $(X,d,\m)$ be an essentially non-branching $CD(K,N)$  space with $\supp\m=X$, $\m(X)<\infty$, $K\in \R$ and $N\in (1,\infty)$.

Then, for any $1$-Lipschitz function $u:X\rightarrow \R$ there exists a disintegration $\{\m_{\alpha}\}_{\alpha\in Q}$ of $\m|_{\T_u}$ that is strongly consistent with $R_u$. 

Moreover, there exists $\tilde Q$ such that $\mathfrak q(Q\backslash \tilde Q)=0$ and $\forall \alpha\in \tilde Q$, $\m_{\alpha}$ is a Radon measure with $\m_{\alpha}=h_{\alpha}\mathcal{H}^1|_{X_{\alpha}}$ and $(X_{\alpha}, d, \m_{\alpha})$ verifies the condition $CD(K,N)$.

More precisely, for all $\alpha\in \tilde Q$ it holds that
\begin{align}\label{kuconcave}
h_{\alpha}(\gamma_t)^{\frac{1}{N-1}}\geq \sigma_{K/N-1}^{(1-t)}(|\dot\gamma|)h_{\alpha}(\gamma_0)^{\frac{1}{N-1}}+\sigma_{K/N-1}^{(t)}(|\dot\gamma|)h_{\alpha}(\gamma_1)^{\frac{1}{N-1}}
\end{align}
for every geodesic $\gamma:[0,1]\rightarrow (a(X_\alpha),b(X_\alpha))$.
\end{theorem}
\begin{remark} The property
\eqref{kuconcave} yields that $h_{\alpha}$ is locally Lipschitz continuous on $(a(X_\alpha),b(X_\beta))$ \cite[Section 4]{cavmon}, and that $h_{\alpha}:\R \rightarrow (0,\infty)$ satifies
\begin{align*}
\frac{d^2}{dr^2}h_{\alpha}^{\frac{1}{N-1}}+ \frac{K}{N-1}h_{\alpha}^{\frac{1}{N-1}}\leq 0 \mbox{ on $(a(X_{\alpha}),b(X_{\alpha}))$} \mbox{ in distributional sense.}
\end{align*}
\end{remark}
\begin{remark} The Bishop-Gromov volume monotonicity implies that $h_\alpha$ can always be extended to continuous function on $[a(X_\alpha),b(X_\alpha)]$ \cite[Remark 2.14]{cav-mon-lapl-18}. Then \eqref{kuconcave} holds for every geodesic $\gamma:[0,1]\rightarrow [a(X_\alpha),b(X_\alpha)]$.
We set $\left(h_{\alpha}\circ \gamma_{\alpha}(r)\right)\cdot1_{[a(X_{\alpha}),b(X_{\alpha})]}=h_{\alpha}(r)$
and consider $h_{\alpha}$ as function that is defined everywhere on $\R$. 
We also consider $h_{\alpha}': X_{\alpha}\rightarrow \R$ defined via 
$h_{\alpha}'(\gamma_{\alpha}(r))=h_{\alpha}'(r)$. 
\end{remark}
\begin{remark}\label{dagger}
In the following we set $Q^\dagger:= \tilde Q \cap \hat Q$. Then, $\mathfrak q(Q\backslash Q^\dagger)=0$ and  for every $\alpha \in Q^\dagger$ the inequality \eqref{kuconcave} and \eqref{somehow} hold. We also set $\mathfrak Q^{-1}(Q^{\dagger})=:\T_u^{\dagger}\subset \T_u$ and $\bigcup_{x\in \T_u^{\dagger}} R_u(x)=:\T_{u,e}^{\dagger}\subset \T_{u,e}$. 
\end{remark}
\section{$1$-dimensional comparison results}
\noindent
Let $u:[0,\theta]\rightarrow (0,\infty)$ 
%
such that
\begin{align}\label{kuconcavity}
u\circ\gamma(t)\geq \sigma_{\kappa}^{(1-t)}(|\dot{\gamma}|)u\circ \gamma(0) + \sigma_{\kappa}^{(t)}(|\dot{\gamma}|)u\circ\gamma(1)
\end{align}
for any constant speed geodesic $\gamma:[0,1]\rightarrow [0,\theta]$.
%
%
%
Then $u$ is semi-concave and therefore locally Lipschitz on $(0,\theta)$. $u$ satisfies $u''+\kappa u\leq 0$ in distributional sense \cite[Lemma 2.8]{erbarkuwadasturm}.
The  limits
\begin{align*}
\frac{d^+}{dr} u (r)=\lim_{h\downarrow 0} \frac{u(r+h)-u(r)}{h} \in \R\cup\{\infty\} \ \& \
\frac{d^-}{dr} u (r)=\lim_{h\downarrow 0} \frac{u(r-h)-u(r)}{-h}\in \R\cup\{-\infty\}
\end{align*}
exist for every $r\in [0,\theta]$ and are in $\R$ for $r\in (0,\theta)$. If $\frac{d^+}{dr} u(0)<\infty$ then $u$ is continuous in $0$. 
Moreover, $\frac{d^{+/-}}{dr}u$ is continuous from the right/left on $(0,\theta)$, and 
\begin{align}\label{ineq:derivatives}
\frac{d^+}{dr}u (r)\leq
\frac{d^-}{dr} u (r)
\end{align}
with equality if and only if $u$ is differentiable in $r\in (0,\theta)$. 
In particular $u$ is locally semi-concave, and $u$ is twice differentiable $\mathcal L^1$-almost everywhere. 
%
%
%
\begin{lemma}\label{lem:comparison}
Let $u:[0,\theta]\rightarrow (0,\infty)$ be as above. Let $r_0\in (0,\theta)$.
Then 
\begin{align*}
u(r)\leq u(r_0)\cos_{\kappa}(r-r_0) + \frac{d^+}{dr}u(r_0)\sin_\kappa(r-r_0) \ \mbox{ on }(r_0,\theta).
\end{align*}
In particular, the right hand side is positive on $(r_0,\theta)$.
\end{lemma}
\begin{proof}
Consider $\phi\in C^{\infty}_0((-1,1))$ with $\int_{-1}^1\phi(t)dt=1$ and $\phi_{\epsilon}(t)=\frac{1}{\epsilon}\phi(\frac{t}{\epsilon})$.
We set
\begin{align*}
\tilde{u}(s)=u\star\phi_{\epsilon}(s)=\int_{-\epsilon}^\epsilon\phi_{\epsilon}(-r)u(s-r)dr=\int_{s-\epsilon}^{s+\epsilon}\phi_{\epsilon}(t-s)u(t)dr
\end{align*}
for $s\in (\epsilon, \theta-\epsilon)$. We choose $\epsilon>0$ small enough such that $r_0\in (\epsilon,\theta-\epsilon)$. Then
\begin{align*}
\tilde{u}''(s)=(u\star \phi_{\epsilon})''(s)
=\int_0^\theta\phi_{\epsilon}''(t-s)u(t)dt
&\leq -\VK\int_{-\epsilon}^{\epsilon}\phi_{\epsilon}(-r)u(r-s)dr\leq -\VK \tilde{u}(s).
\end{align*}
Hence, by  classical Sturm comparison \cite{docarmo} we obtain
\begin{align*}
\tilde{u}(r)\leq \tilde{u}(r_0)\cos_{\kappa}(r-r_0)+ \tilde{u}'(r_0)\sin_{\kappa}(r-r_0)\  \mbox{ on }(r_0,\theta-\epsilon).
\end{align*}
Now, one can check that $\tilde{u}(r)=\phi_{\epsilon}\star u(r)\rightarrow u(r)$ on $(\epsilon_0,\theta-\epsilon_0)$ if $\epsilon\in (0,\epsilon_0)$ and $\epsilon\rightarrow 0$, and also
$$
\tilde{u}'(r_0)=\frac{d^+}{dr}u(r_0)= \int_{-\epsilon}^0 \phi(-r)\frac{d^+}{ds}[u(s-r)]_{s=r_0} dr= \phi_{\epsilon}\star \frac{d^+}{dr}u(r_0)\rightarrow  \frac{d^+}{dr}u(r_0).
$$
Hence, we obtain that 
\begin{align*}
u(r)\leq u(r_0)\cos_{\kappa}(r-r_0)+ \frac{d^+}{dr} u(r_0)\sin_{\kappa}(r-r_0)\ \mbox{ for } r\in (r_0,\theta-\epsilon_0).
\end{align*}
 Finally, since we can choose $\epsilon_0>0$ arbitrarily small, we obtain the result.
\end{proof}
%

\begin{definition}\label{jacobian}
Let $K\in \mathbb{R}$, $H\in (-\infty,\infty)$, $N> 1$. The Jacobian function is defined as 
\begin{align*}
r\in \mathbb{R}\mapsto J_{H,K,N}(t)=\left(\cos_{K/{N-1}}(r)+\frac{H}{N-1}\sin_{K/(N-1)}(r)\right)_+^{N-1}.
\end{align*}
$J_{H,K,N}$ is pointwise monotone non-decreasing in $H$ and $K$, and monotone non-increasing in $N$.
\end{definition}  

\begin{corollary}\label{cor:key}{
Let $h:(a,b)\rightarrow (0,\infty)$ such that $a<0<b$ and }
for any constant speed geodesic $\gamma:[0,1]\rightarrow (a,b)$ it holds
\begin{align*}
h(\gamma_t)^{\frac{1}{N-1}}\geq \sigma_{K/N-1}^{(1-t)}(|\dot\gamma|)h(\gamma_0)^{\frac{1}{N-1}}+\sigma_{K/N-1}^{(t)}(|\dot\gamma|)h(\gamma_1)^{\frac{1}{N-1}}.
\end{align*}
Then
$
h(r)h(0)^{-1}\leq J_{K,H,N}(r) \mbox{ for }r\in (0,b)
$
where 
$
H=
\frac{d^{+}}{dr}\log h(0).
$
\end{corollary}
\begin{proof}
Applyng Lemma \ref{lem:comparison} to $h^{\frac{1}{N-1}}$ yields
\begin{align*}
h^{\frac{1}{N-1}}(r)\leq h^{\frac{1}{N-1}}(0) \cos_{K/(N-1)}(r) + \frac{d^+}{dr} h^{\frac{1}{N-1}}(0) \sin_{K/(N-1)}(r).
\end{align*}
Now, we note that $\frac{d^+}{dr} h^{\frac{1}{N-1}}(0)=\frac{1}{N-1} h^{\frac{1}{N-1}}(0)\frac{d^+}{dt} \log h(0)$, we devide by $h^{\frac{1}{N-1}}(0)$ and apply $(\cdot)^{N-1}$ to both sides.
\end{proof}
\section{Mean curvature in the context of $CD(K,N)$ spaces.}\label{sec:mean}
Let $(X,d,\m)$ be a metric measure space as in Theorem \ref{th:1dlocalisation}.
Let $\Omega\subset X$ be a closed subset, and let $S=\partial \Omega$ such that $\m(S)=0$. The function $d_\Omega: X\rightarrow \mathbb{R}$ is given by 
\begin{align*}
\inf_{y\in \bar{\Omega}} d(x,y)=: d_\Omega(x).
\end{align*}
Let us also define $d_{\Omega}^*:=d_{\overline{\Omega^c}}$. The signed distance function $d_S$ for $S$ is  given by 
\begin{align*}
d_S=d_{\Omega}- d_{\Omega}^*: X\rightarrow \mathbb{R}.
\end{align*}
It follows that $d_S(x)=0$ if and only if $x\in S$, $d_S\leq 0$ if $x\in \Omega$ and $d_S\geq 0$ if $x\in \Omega^c$. 
It is clear that $d_S|_{\Omega}= -d_{\Omega}^*$ and $d_S|_{\Omega^c}=d_\Omega$.
Setting $v=d_S$ we can also write 
$$
d_S(x)= \sign(v(x))d(\{v=0\},x), \forall x\in X.
$$
$d_S$ is $1$-Lipschitz.  $\Omega^\circ$ denote the topological interior of $\Omega$.

%
Let $\mathcal{T}_{d_S,e}$ be the transport set of $d_S$ with endpoints. We have $\mathcal{T}_{d_S, e}\supset X\backslash S$. In particular, we have $\m(X\backslash \T_{d_S})=0$ by Lemma \ref{somelemma}.

Therefore, the $1$-Lipschitz function $d_S$ induces a partition $\left\{X_{\alpha}\right\}_{\alpha\in Q}$ of $X$ up to a set of measure zero for a measurable quotient space $Q$, and a disintegration $\{m_{\alpha}\}_{\alpha\in Q}$ that is strongly consistent with the partition. The subset $X_{\alpha}$, $\alpha \in Q$, is the image of a geodesic $\gamma_{\alpha}:I_\alpha\rightarrow X$. 

We consider $Q^\dagger\subset Q$ as in Remark \ref{dagger}. One has the representation 
$$\m(B)=\int_{Q} \m_{\alpha}(B) d\mathfrak q(\alpha)=\int_{Q^{\dagger}} \int_{\gamma_{\alpha}^{-1}(B)} h_{\alpha}(r) dr d\mathfrak q(\alpha)\ \ \forall B\in \mathcal B.$$ 
For any transport ray $X_{\alpha}$, $\alpha\in Q^{\dagger}$, it holds that
$d_S(\gamma_\alpha(b(X_{\alpha})))\geq 0$ and  $d_S(\gamma_\alpha(a(X_{\alpha})))\leq 0.$
\begin{remark}

It is easy to see that $A:=\mathfrak Q^{-1}(\mathfrak Q(S\cap \T_{d_S}))\subset \T_{d_S}$ is a measurable subset. 
The set
$A\subset \T_u$ is defined such that  $\forall \alpha \in \mathfrak Q(A)$ we have $X_\alpha=\gamma_{\alpha}((a(X_\alpha),b(X_\alpha))\cap S =\{\gamma(t_\alpha)\}\neq \emptyset$ for a unique $t_\alpha\in (a(X_\alpha),b(X_\alpha))$. Then, the map $\gamma(t)\in A \mapsto \gamma(t_\alpha)\in S\cap \T_u$ is a measurable section on $A\subset \T_{d_S}$, one can  identify the measurable set $\mathfrak Q(A)\subset Q$ with $A\cap S$ and one can parametrize $\gamma_\alpha$ such that $t_\alpha=0$. 
Moreover, we define
\begin{align*}
\mbox{$
A\cap \mathfrak \T_u^{\dagger}=: A^{\dagger} \ \mbox{ and }\ \bigcup_{x\in A^{\dagger}}R_u(x)=:A_e^{\dagger}.$}
\end{align*}
The sets $A^\dagger$ and $A_e^\dagger$ are measurable, and also 
\begin{align*}\mbox{$
B_{in}^\dagger:=\Omega^\circ \cap \T^\dagger_{d_S} \backslash A^\dagger\subset \T^\dagger_{d_S}$ \ and \ $B^\dagger_{out}:=\Omega^c \cap \T^\dagger_{d_S} \backslash A^\dagger\subset \T_{d_S}$}
\end{align*}
as well as $\bigcup_{x\in B^\dagger_{out}} R_u(x)=B^\dagger_{out, e}$ and $\bigcup_{x\in B^\dagger_{in}}R_u(x)=:B^\dagger_{in,e}$
are measurable. 
The map $\alpha \in  \mathfrak Q(A^\dagger)\mapsto h_\alpha(\gamma_\alpha(0))\in \R$ is  measurable (see \cite[Proposition 10.4]{cavmil}).
\end{remark}

\begin{definition}\label{def:surfacemeasure}
We define a {\it surface measure} $\m_S$ via
$$
\int\phi(x)d\m_S(x):=\int_{\mathfrak Q(A^{\dagger})} \phi(\gamma_{\alpha}(0)) h_{\alpha}(0) d\mathfrak q(\alpha)
$$
for  any continuous function $\phi:X\rightarrow \R$. That is $\m_S$ is the pushforward of $h_\alpha(\gamma_\alpha(0))\mathfrak q(d\alpha)$ under the map $\gamma\in \mathfrak Q(A^{\dagger})\mapsto \gamma(0)$.
\end{definition}
\begin{remark}
We note that the measure $\m_S$ is by definition concentrated on $S\cap A^{\dagger}$.
\end{remark}
\begin{remark}\label{rem:smooth}
Let us adress briefly the smooth case. 
Let $(M,g,\Psi \vol_g)$ be compact weighted Riemannian manifold.
Let $S\subset M$ (with $S=\partial \Omega$ for $\Omega\in \mathcal B(M)$). Assume that $S$ is an $(n-1)$-dimensional compact $C^2$-submanifold. Then, the signed distance function $d_S$ is smooth on a neighborhood $U$ of $S$ and $\nabla d_S$ is the smooth unit normal vectorfield along $S$. More precisely, $\nabla d_S(x)\perp T_xS$ and $|\nabla d_S(x)|=1$ for all $x\in S$. We denote $\vol_S$ the induced volume for $S$.

Recall that for every $x\in S$ there exist $a_x<0$ and $b_x>0$ such that $\gamma_x(r)=\exp_x(r\nabla d_S(x))$ is a minimal geodesic on $(a_x,b_x)\subset \R$, and we define
\begin{align*}
\mathcal U=\left\{(x,r)\in S\times \R: r\in (\alpha_x,b_x)\right\}\subset S\times \R
\end{align*}
and the map $T:\mathcal U \rightarrow M$ via $T(x,r)=\gamma_x(r)$. It is well-known that $T$ is a diffeomorphism on $\mathcal U$, that $\vol_g(M\backslash T(\mathcal U))=0$ and that integrals can be computed effectively by the following formula:
\begin{align}\label{effectiv}
\int g d\m 
 = \int_S \int_{a_x}^{b_x} g\circ T(x,r) \det DT_{(x,r)}|_{T_xS} \Psi\circ T(x,r) drd\!\vol_S(x).
\end{align} 
On the other hand, we can define a map $\mathfrak Q: T(\mathcal U)\rightarrow S$ via $\mathfrak Q= \mbox{Pr}\circ T^{-1}$ where $\mbox{Pr}: \mathcal U\rightarrow S$ is the projection map.
Then $\mathfrak Q^{-1}(x)=\gamma_x:(a_x,b_x)\rightarrow M$, $x\in S$, are precisely the non-branched transport geodesics w.r.t. $d_S$, $\mathfrak Q^{-1}(S)=\T_{d_S}^b$ and $(a_x,b_x)=(a(X_x),b(X_x))$. Moreover, we see that $$\mathfrak q=\mathfrak Q_{\#}\m=\underbrace{\left[\int_{a_x}^{b_x} \det DT_{(x,r)}\Psi\circ T(x,r) dr\right]}_{=:f(x)} \vol_S(dx).$$ Hence, in this case we can identify $S$ with $Q^\dagger$, and the quotient measure $\mathfrak q$ on $S$ with $\mathfrak q=f(x)\vol_S(dx)$. The integration formula \eqref{effectiv} becomes
\begin{align}\label{effectiv2}
\int g d\m
 = \int_S \frac{1}{f(x)}\int_{a_x}^{b_x} g\circ T(x,r) \det DT_{(x,r)}|_{T_xS}\Psi\circ T(x,r) dr d\mathfrak q(x).
\end{align}
By the uniqueness statement in the disintegration theorem and by \eqref{effectiv2} we therefore have that $h_x(r)= \frac{1}{f(x)}\det DT_{(x,r)}|_{T_{(x,r)}} \Psi \circ T(x,r)S$ and $h_x(0)= \frac{1}{f(x)} \Psi(x)$.

It follows that for a measurable set $B\subset M$ that
$$
\int _{S\cap B}\Psi d\vol_S= \lim_{t\rightarrow 0} \int_{\mathfrak Q(B)} \frac{1}{t}\int_0^{t}h_\alpha(r)dr d\mathfrak q(\alpha)=\m_S(B).
$$
Hence, the measure $\m_S$ coincides with $\Psi d\vol_S$ in this case.
\end{remark}

Let us recall another result of Cavalletti-Mondino.
\begin{theorem}[\cite{cav-mon-lapl-18}]\label{thm:cm} Let $(X,d,\m)$ be an $CD(K,N)$ space, and $\Omega$ and $S=\partial \Omega$ as above. Then
$d_S\in D({\bf \Delta}, X\backslash S)$, and one element of ${\bf \Delta}d_S|_{X\backslash S}$ that we denote with $\Delta d_S|_{X\backslash S}$ is the Radon functional on $X\backslash S$ given by the representation formula
\begin{align*}
\Delta d_S|_{X\backslash S}=- (\log h_{\alpha})'\m|_{X\backslash S}- \int_Q ( h_{\alpha}\delta_{a(X_{\alpha})\cap \{d_S>0\}} - h_{\alpha}\delta_{b(X_{\alpha})\cap \{d_S<0\}} ) d\mathfrak q(\alpha).
\end{align*}
We note that the Radon functional $\Delta d_S|_{X\backslash S}$ can be represented as the difference of two measures $[\Delta d_S]^+$ and $[\Delta d_S|_{X\backslash S}]^-$ such that 
\begin{align*}
[\Delta d_S|_{X\backslash S}]^+_{reg} - [\Delta d_S|_{X\backslash S}]^-_{reg} = - (\log h_{\alpha})' \ \ \m\mbox{-a.e.}
\end{align*}
where $[\Delta d_S|_{X\backslash S}]^{\pm}_{reg}$ denotes the $\m$-absolutely continuous part in the Lebesgue decomposition of $[\Delta d_S|_{X\backslash S}]^{\pm}$. In particular, $-(\log h_{\alpha})'$ coincides with a measurable function $\m$-a.e.\ .
\end{theorem}
\begin{remark}
In the light of the previous section and since $h_{\alpha}^{\frac{1}{N-1}}$ is semiconcave on $(a(X_\alpha),b(X_\beta))$, $-(\log h_{\alpha})'$ coincides $\m$-a.e. with the function {$\frac{1}{h_{\alpha}}\frac{d^{+/-}}{dr}h_{\alpha}: X\rightarrow \R$} that is defined via
\begin{align*}
p\in \T_{u}^{\dagger}\mapsto \frac{d^{+/-}}{dr}h_{\alpha}(\gamma_{\alpha}(r)) \ \mbox{ if }p=\gamma_{\alpha}(r)\mbox{ for }r\in (a(X_\alpha),b(X_{\alpha})).
\end{align*}
Hence,  $\frac{d^{+}}{dr}h_\alpha\ \&\ \frac{d^{-}}{dr}h_\alpha$  are measurable functions on $X$ and everywhere defined on $\T_{d_S}^{\dagger}$.
\end{remark}
\begin{definition}\label{def:meancurvature}
Set $S=\partial \Omega$ and let $\{X_{q}\}_{q\in Q}$ be the induced disintegration. 

We say that $S$ has finite outer curvature if $\m(B^\dagger_{out})=0$, $S$ has finite inner curvature if $\m(B^\dagger_{in})=0$, and $S$ has finite curvature if $\m(B^\dagger_{out}\cup B^\dagger_{in})=0$.

Provided $S$ has finite outer curvature we define the outer mean curvature of $S$ as 
\begin{align*}
p\in S\mapsto 
H^+(p)= \begin{cases}
\frac{d^+}{dr}\log h_\alpha(\gamma_\alpha(0))
& \mbox{ if }
\ p=\gamma_\alpha(0)\in S\cap A^{\dagger}\\
-\infty&  \mbox{ if } \ p\in B^\dagger_{in, e}\cap S\\
c \mbox{ for some }c\in \R & \mbox{ otherwise.}
\end{cases}
\end{align*}
If we switch the roles of $\Omega$ and $\overline{\Omega^c}$ and $S$ has finite inner  curvature, then we call the corresponding outer mean curvature the inner mean curvature and we write $H^-$. 

For $S$ with finite curvature the mean curvature is defined as $\max\{ H^+, -H^-\}=:H$.
\end{definition}

\begin{remark}\label{rem:classicalH} Let us again go back to the smooth situation of Remark \ref{rem:smooth}.
In this case $r\mapsto h_{\alpha}(r)=\det DT_{(\alpha,r)}|_{T_{(\alpha,r)}S}$, $\alpha\in S$, is smooth on the maximal open interval $(a(X_\alpha),b(X_\alpha))$ where $\gamma_\alpha$ is a geodesic. 
%
Moreover, $T:\mathcal U\rightarrow M$ is a smooth map.  Hence
\begin{align*}
\frac{d}{dr}\Big|_{0} \log h_\alpha(r) &= \mbox{tr}^{T_{(\alpha,r)}S} \frac{d}{dr}\Big|_0 DT_{(\alpha,r)}|_{T_{(\alpha,r)}S} =- \mbox{Div}^{T_{(\alpha,r)}S} \nabla d_{S}(\alpha)=- \langle {\bf H}(\alpha),\nabla d_S(\alpha)\rangle 
\end{align*}
for 
$\alpha\in S$
where ${\bf H}=H \nabla d_S$ denotes the mean curvature vector along $S$. We conclude that in this case our notion of mean curvature coincides with the classical one.
\end{remark}

\begin{remark}\label{rem:ballcondition}
As was pointed out to the author by the referee the property of having finite outer/inner curvature in the sense of the previous definition corresponds to an interior/exterior ball condition for $\Omega$ what is a condition on the full second fundamental form of the boundary $\partial \Omega$ in smooth context.
\end{remark}

\begin{remark}
The definition of mean curvature as given in Definition \ref{def:meancurvature} is sufficient for the Heintze-Karcher theorem. But considering Theorem \ref{thm:cm} one might define the mean curvature of $S=\partial \Omega$ as the measure given by
\begin{align*}
H|_{S\cap A^\dagger} d\m_S + \int_{\mathfrak Q(B^\dagger_{out})} h_{\alpha}\delta_{a(X_{\alpha})\cap S}d\mathfrak q(\alpha) -\int_{\mathfrak Q(B^\dagger_{in})} h_{\alpha}\delta_{b(X_{\alpha})\cap S} d\mathfrak q(\alpha).
\end{align*}
\end{remark}
%
%
\section{Proof of the main theorems}
{\it Proof of Theorem \ref{main}}.
Let $\Omega\subset X$ be closed subset, $S=\partial \Omega$ and $d_S$ as before.
Consider 
\begin{align*}
S^+_t= B_t(\Omega)\backslash \Omega \ \ \&\ \ 
S^-_t= B_t(\Omega)\backslash \overline{\Omega}^c,
\end{align*}
where  $B_t(\Omega)=\left\{x\in X: \exists y\in \Omega \mbox{ s.t. }x\in B_t(y)\right\}$.
One has $(X_{\alpha}, d)\equiv [a(X_\alpha),b(X_\alpha)]$ via $\gamma_{\alpha}$.
One can check that
 $$S^+_t\cap X_\alpha\equiv [0,b(X_\alpha)\wedge t],\ \ \ S^-_t\cap X_\alpha\equiv [a(X_\alpha)\vee -t,0],\ \ \ \forall t\in (0,\infty).$$
We assume finite outer curvature, that is $\m(B_{out})=0$. Moreover, $B^\dagger_{in}\subset \Omega^\circ$ and therefore $\m(S^+_t\cap B^\dagger_{in})=0$. Let $D=\diam_X$.
%
 Theorem \ref{th:1dlocalisation} ($1D$-localisation) and Corollary \ref{cor:key} yield
\begin{align*}
\m(S_t^+)=&\m(S_t^+\cap \T_u^\dagger)=  \m((S_t^+\cap \T^\dagger_u)\backslash (B^\dagger_{out}\cup B^\dagger_{in}))= \int_{\mathfrak Q(A^{\dagger})} \int_{S_t^+\cap X_\alpha} h_\alpha (r) dr d\mathfrak q(\alpha)\\
 \leq & \int_{\mathfrak Q(A^{\dagger})} \int_0^{t} J_{H^+(\gamma_{\alpha}(0)),K,N}(r) dr h_{\alpha}(0) d\mathfrak q(\alpha)
\leq  \int \int_0^t J_{H^+(p),K,N}(r)dr d\!\m_S(p).
\end{align*}
This is the first claim in Theorem \ref{main}.

Now, we assume finite  curvature.
By switching the roles of $\Omega$ and $\Omega^c$ we obtain similarly
\begin{align*}
\m(S_t^-)
&\leq \int\int_0^t J_{H^-(p),K,N}(r)dr d\!\m_S(p)\leq  \int \int_0^{D} J_{H^-(p),K,N}(r)dr d\!\m_S(p)\\
&\ \ \ \ \ \ \ \ \ \ = \int \int_{-D}^0 J_{-H^-(p),K,N}(r)dr d\!\m_S(p)\leq \int \int_{-D}^0 J_{H(p),K,N}(r)dr d\!\m_S(p).
\end{align*}
Note that by the symmetries of $\sin_{K/(N-1)}$ and $\cos_{K/(N-1)}$ we have that $J_{-H,K,N}(r)=J_{H,K,N}(-r)$. 
In the last inequality we use $H=\max\{H^+, -H^-\}$.
Hence
\begin{align*}
\m(X)= \m(S^-_D)+\m(S^+_D)\leq\int\int_{-D}^D J_{H(p),K,N}(r) dr d\!\m_S(p)
\end{align*}
This proves Theorem \ref{main}. \qed

{\it Proof of Corollary \ref{corollary}}. Let us prove the second claim of the corollary. The first one follows from Theorem \ref{main}.  
Assume $K\geq 0$, $S$ has finite curvature and $H\leq 0$. Then $J_{H,0,N}(r)\leq 1$ on $(0,\infty)$ and from the proof of the main theorem we get that
\begin{align*}
\m(S_t^+\cup S_t^-)\leq \int t\wedge\left[ b(X_\alpha) -a(X_\alpha)\right] h_\alpha(0) d\mathfrak q(\alpha) \leq \diam_X \m_S(S).
\end{align*}
When we let $t\rightarrow \infty$, the claim follows. 

{\it Proof of Corollary \ref{corollary1}}.  Let $K>0$. 
Consider 
$r\in I\mapsto f(r)=\cos_{K/(N-1)}(r)+\frac{H}{N-1}\sin_{K/(N-1)}(r)$
where $I$ is the connected component of $\overline{\{f(r)>0\}}$ that contains $0\in \R$. $f$ solves $f''+\frac{K}{N-1}f=0$ on $I$ and a straighforward computation yields
$$
(f')^2+\frac{K}{N-1} f^2 = \frac{K}{N-1}+\left(\frac{H}{N-1}\right)^2.
$$
We set $\kappa:= \frac{K}{N-1}+\left(\frac{H}{N-1}\right)^2$. We can see that up to translation $f:I\rightarrow [0,\infty)$ must coincide with $\sqrt{\kappa}\sin_{K/(N-1)}:[0,\pi_{K/(N-1)}]\rightarrow [0,\infty)$.
Hence
\begin{align*}
\int_\R J_{K,H(p),N}(r) dr = (\sqrt{\kappa})^{N-1} \int_0^{\pi_{K/(N-1)}}\sin_{K/(N-1)}^{N-1}(r) dr.
\end{align*}
We can plug this back into the Heintze-Karcher inequality \eqref{ineq:hk} and obtain Corollary \ref{corollary1}. \qed


{\it Proof of Theorem \ref{main2}}. 
Assume $K>0$ and equality in the Heintze-Karcher estimate \eqref{ineq:hk}, or equivalently assume equality in Corollary \ref{corollary1}.

Then, all the inequalities in the proof before become equalities. In particular, from Corollary \ref{cor:key} we obtain that 
$$h_{\alpha}(r)= h_{\alpha}(0) J_{H(\gamma_{\alpha}(0)),K,N}(r) \mbox{ on } [a(X_{\alpha}),b(X_\alpha)]\ \ \forall \alpha\in \mathfrak Q(A^\dagger).$$
Recall that $h_\alpha(r) dr=\m_\alpha$ is a probability measure for $\alpha\in Q^\dagger$ by construction of the disintegration.
Plugging that back into the Heintze-Karcher inequality yields
\begin{align*}
\m(\Omega)\cup \m(S^+_t)=\m(B_t(\Omega))= \int\int_{-\infty}^t J_{H(p),K,N}(r) dr d\!\m_{S}(p) \ \ \forall t>0.
\end{align*}
We consider $\bar \m=\m(X)^{-1}\m$. The corresponding Minkowski content computes as
\begin{align*}
\bar m^+(\partial \Omega)&=\int_Q J_{H(\gamma_{\alpha}(0)),K,N}(0)h_\alpha(0)d\bar{\mathfrak q}(\alpha)= \int_Q \mathcal{I}_{K,N,\infty}(v_\alpha) d\bar{\mathfrak q} (\alpha)
\end{align*}
where $\bar{\mathfrak q}=\m(X)^{-1}\mathfrak q$ is a probability measure and $v_{\alpha}:= \int_{-\infty}^0 J_{H(\gamma_{\alpha}(0)),K,N}(r) h_\alpha(0) dr=\m_{\alpha}((a(X_\alpha),0))$. We also observe that 
\begin{align*}
\int_Q \int_{-\infty}^0 J_{H(\gamma_\alpha(0)),K,N}(r) h_\alpha(0) dr d\bar{\mathfrak q}(\alpha)= \int_Q v_{\alpha} d\bar{\mathfrak q}(\alpha)=\int_Q\frac{\m_\alpha((a(X_\alpha),0)}{\m(X)} {d\mathfrak q(\alpha)} =\frac{\m(\Omega)}{\m(X)}.
\end{align*}
We set $f(t)=\frac{1}{c}\int_0^t\sin_{K/(N-1)}^{N-1}(r) dr$. $f^{-1}: [0, 1]\rightarrow [0,\infty)$ exists and is monotone nondecreasing where $c=\int_0^{\pi_{K/(N-1)}} \sin_{K/(N-1)}(r)dr$. One can check that $v\in [0,1]\mapsto \mathcal I_{K,N,\infty}(v)=f'\circ f^{-1}(v)=:h(v)$.
Moreover, we compute
\begin{align*}
h'(v)= \cos_{K/(N-1)}\circ f^{-1} (v)\left[ f' \circ f^{-1}(v)\right]= \frac{\cos_{K/(N-1)}}{\sin_{K/(N-1)}} \circ f^{-1}(v).
\end{align*}
We see that $h'$ is monotone nonincreasing, hence $h$ is concave.
Jensen's inequality yields that
\begin{align*}
\bar\m^+(\partial \Omega)&= \int_Q \mathcal{I}_{K,N,\infty}(v_\alpha) d\bar{\mathfrak q}(\alpha)\leq  \mathcal{I}_{K,N,\infty}\left(\int v_{\alpha} d\bar{\mathfrak q}(\alpha)\right) = \mathcal I_{K,N,\infty}( \bar\m(\Omega)).
\end{align*}
Hence by Theorem \ref{th:cavmon} there is equality and Theorem \ref{th:cavmon} yields the result. \qed
\small{
\bibliography{new}

\bibliographystyle{amsalpha}}
\end{document}